\documentclass[smallextended]{svjour3}

\smartqed

\journalname{Optimization Letters}

\usepackage[T1]{fontenc}
\usepackage[utf8]{inputenc}
\usepackage[english]{babel}
\usepackage{lmodern}

\usepackage{amsmath}
\usepackage{amssymb}
\usepackage{bm}
\usepackage{bbm}
\usepackage{spalign}

\usepackage{fix-cm}

\usepackage[obeyspaces]{url}

\usepackage{graphicx}
\usepackage{float}

\usepackage{fancyvrb}
\usepackage{empheq}

\usepackage{tikz}
\usepackage{tikzit}

\makeatletter
\let\cl@chapter\undefined
\makeatletter

\usepackage[capitalize,nameinlink]{cleveref}[0.21]

\tikzstyle{node_style}=[fill=white, draw={rgb,255: red,1; green,128; blue,71}, shape=circle]
\tikzstyle{node_intermediate}=[fill=white, draw={rgb,255: red,255; green,128; blue,0}, shape=circle]
\tikzstyle{node_lower}=[fill=white, draw={rgb,255: red,191; green,0; blue,64}, shape=circle]

\tikzstyle{edge_anticipate}=[draw={rgb,255: red,0; green,207; blue,207}, <-, dash pattern=on 3pt off 2pt, line width=2pt]
\tikzstyle{edge_param}=[draw={rgb,255: red,190; green,0; blue,0}, ->, line width=2pt]

\title{Complexity of near-optimal robust versions of multilevel optimization problems}
\author{Mathieu Besançon, Miguel F. Anjos, Luce Brotcorne}

\institute{Mathieu Besançon \at
              Department of Mathematics and Industrial Engineering\\
              Polytechnique Montréal, QC, Canada \&\\
              Centrale Lille \&\\
              INRIA Lille Nord-Europe, France \\
              \email{mathieu.besancon@polymtl.ca}           
            \and
            Miguel F. Anjos \at
              School of Mathematics, University of Edinburgh, UK
            \and
            Luce Brotcorne \at
              INRIA Lille Nord-Europe, France
}

\DeclareMathOperator*{\argmax}{arg\,max}
\DeclareMathOperator*{\argmin}{arg\,min}

\begin{document}

\maketitle

\begin{abstract}
Near-optimality robustness extends multilevel optimization with a
limited deviation of a lower level from its optimal solution,
anticipated by higher levels. We analyze the complexity of near-optimal
robust multilevel problems, where near-optimal robustness is modelled through
additional adversarial decision-makers. Near-optimal robust versions
of multilevel problems are shown to remain in the same complexity class
as the problem without near-optimality robustness under general conditions.
\end{abstract}

\section{Introduction}

Multilevel optimization is a class of mathematical optimization problems
where other problems are embedded in the constraints.
They are well suited to model sequential decision-making processes,
where a first decision-maker,
the leader intrinsically integrates the reaction of another
decision-maker, the follower, into their decision-making problem.

In recent years, most of the research focuses on the study and design of
efficient solution methods for the case of two levels, namely bilevel problems \cite{dempe2018review},
which fostered a growing range of applications.
\\

Near-optimal robustness, defined in \cite{besanccon2019near},
is an extension of bilevel optimization. In this setting,
the upper level anticipates limited deviations of the lower level
from an optimal solution and aims at a solution that remains feasible
for any feasible and near-optimal solution of the lower level. 
This protection of the upper level against uncertain deviations of the lower-level
has led to the characterization of near-optimality robustness as a
robust optimization approach for bilevel optimization.
The lower-level response corresponds to the uncertain parameter
and the maximum deviation of the objective value from an optimal solution
to the uncertainty budget.
Because the set of near-optimal lower-level solutions potentially has
infinite cardinality and depends on the upper-level decision itself,
near-optimality robustness adds generalized semi-infinite constraints
to the bilevel problem. The additional constraint can also be viewed as
a form of robustness under decision-dependent uncertainty.\\

In this paper, we prove complexity results on multilevel problems
to which near-optimality robustness constraints are added under various 
forms. We show that under fairly general conditions, the near-optimal
robust version of a multilevel problem remains on the same level of the
polynomial hierarchy as the original problem.
These results are non-trivial assuming that the polynomial hierarchy
does not collapse and open the possibility of solution algorithms
for near-optimal robust multilevel problems as efficient
as for their canonical counterpart.
Even though we focus on near-optimal robust multilevel problems,
the complexity results we establish hold for all multilevel problems
that present the same hierarchical structure, i.e. the same anticipation
and parameterization between levels as the near-optimal formulation
with the adversarial problems, as defined in \cref{sec:simplebilevel}.
\\

The rest of this paper is organized as follows.
\cref{sec:notation} introduces the notation and the background
on near-optimality robustness and existing complexity results in multilevel optimization.
\cref{sec:simplebilevel} presents complexity results for the
near-optimal robust version of bilevel problems, where the
lower level belongs to $\mathcal{P}$ and $\mathcal{NP}$.
These results are extended in \cref{sec:multilevel} to multilevel
optimization problems, focusing on integer multilevel linear problems
with near-optimal deviations of the topmost intermediate level.
\cref{sec:genmulti} provides complexity results for a generalized
form of near-optimal robustness in integer multilevel problems, where
multiple decision-makers anticipate near-optimal reactions of a lower
level. Finally, we draw some conclusions in \cref{sec:conclusioncomplex}.

\section{Background on near-optimality robustness and multilevel optimization}\label{sec:notation}

In this section, we introduce the notation and terminology for bilevel
optimization and near-optimality robustness, and
highlight prior complexity results in multilevel optimization.
Let us define a bilevel problem as:
\begin{subequations}\label{prob:bilevelgen}
\begin{align}
\min_{x}\,\,& F(x, v) \\
\text{s.t.}\,\,\, & G_k(x, v) \leq 0 & \forall k \in \left[\![m_u\right]\!]  \\
& x \in \mathcal{X} \\
& \text{where } v \in \argmin_{y\in\mathcal{Y}} \{f(x,y) \text{ s.t. }g_i(x, y) \leq 0 \,\, \forall i \in \left[\![m_l\right]\!] \}.
\end{align}
\end{subequations}

\noindent
We denote by $\mathcal{X}$ and $\mathcal{Y}$ the domain of upper-
and lower-level variables respectively.
We use the convenience notation $\left[\![n\right]\!] = \{1, \dots,n\}$ for a natural $n$.
\\ 

Problem \eqref{prob:bilevelgen} is ill-posed, since multiple
solutions to the lower level may exist \cite[Ch. 1]{dempe2015bilevel}.
Models often rely on additional assumptions to alleviate this ambiguity,
the two most common being the optimistic and pessimistic approaches.
In the optimistic case (BiP), the lower level 
selects an optimal decision that most favours the upper level.
In this setting, the lower-level decision can be taken by the upper level,
as long as it is optimal for the lower-level problem. The upper level
can thus optimize over both $x$ and $v$, leading to:
\begin{subequations}\label{prob:bilevelstandardopt}
\begin{align}
\text{(BiP): } \min_{x,v}\,\,& F(x, v) \\
\text{s.t.}\,\,\,& G_k(x, v) \leq 0 & \forall k \in \left[\![m_u\right]\!]\label{eq:upfeas}  \\
& x \in \mathcal{X} \\
& v \in \argmin_{y\in\mathcal{Y}} \{f(x,y) \text{ s.t. }g_i(x, y) \leq 0 \,\, \forall i \in \left[\![m_l\right]\!] \}\label{eq:lowopt}.
\end{align}
\end{subequations}

\noindent
Constraint \eqref{eq:lowopt} implies that $v$ is feasible for the
lower level and that $f(x,v)$ is the optimal value of the lower-level
problem, parameterized by $x$.

The pessimistic approach assumes that the lower level chooses
an optimal solution that is the worst for the upper-level
objective as in \cite{dempe2018review} or with respect to the
upper-level constraints as in \cite{Wiesemann2013}.
.\\

\noindent
The near-optimal robust version of (BiP) considers that the lower-level solution may
not be optimal but near-optimal with respect to the lower-level objective function.
The tolerance for near-optimality, denoted by $\delta$ is expressed
as a maximum deviation of the  objective value from optimality.
The problem solved at the upper level must integrate this deviation
and protects the feasibility of its constraints for any near-optimal lower-level decision.
The problem is formulated as:
\begin{subequations}\label{prob:norbip}
\begin{align}
\text{(NORBiP): } \min_{x,v}\,\, & F(x, v) \\
\text{s.t.}\,\,\,& \eqref{eq:upfeas}-\eqref{eq:lowopt}\\
& G_k(x, z) \leq 0\,\, \forall z \in \mathcal{Z}(x;\delta) & \forall k \in \left[\![m_u\right]\!]\label{eq:semiinf}\\
& \text{where } \mathcal{Z}(x;\delta) = \{y\in\mathcal{Y}\,\,\text{\large |}\, f(x, y) \leq f(x, v) + \delta, g(x, y) \leq 0\}.
\end{align}
\end{subequations}
\noindent
$\mathcal{Z}(x;\delta)$ denotes the near-optimal set, i.e.
the set of near-optimal lower-level solutions,
depending on both the upper-level decision $x$ and $\delta$.
(NORBiP) is a generalization of the pessimistic bilevel problem
since the latter is both a special case and
a relaxation of (NORBiP) \cite{besanccon2019near}.
We refer to (BiP) as the \textbf{canonical} problem for 
(NORBiP) (or equivalently Problem \eqref{prob:norbip2})
and (NORBiP) as the near-optimal robust version of (BiP).
In the formulation of (NORBiP), the upper-level objective depends
on decision variables of both levels, but is not protected against near-optimal
deviations. A more conservative formulation also protecting
the objective by moving it to the constraints in an epigraph formulation \cite{besanccon2019near}
is given by:

\begin{align*}
\text{(NORBiP-Alt): } \min_{x,v,\tau}\,\, & \tau \\
\text{s.t.}\,\,\,& \eqref{eq:upfeas}-\eqref{eq:lowopt}\\
& G_k(x, z) \leq 0\,\, \forall z \in \mathcal{Z}(x;\delta) & \forall k \in \left[\![m_u\right]\!]\\
& F(x, z) \text{ }\,\leq \tau \text{ } \forall z \in \mathcal{Z}(x;\delta),\\
\end{align*}
\noindent
The optimal values of the three problems are ordered as:
\begin{equation*}
    \text{opt(BiP)} \leq \text{opt(NORBiP)} \leq \text{opt(NORBiP-Alt)}.
\end{equation*}

We next provide a review of complexity results for bilevel
and multilevel optimization problems.
Bilevel problems are $\mathcal{NP}$-hard in general, even when
the objective functions and constraints at both levels are linear \cite{audet1997links}.
When the lower-level problem is convex, a common solution approach
consists in replacing it with its KKT conditions
\cite{allende2013solving,dempe2019solution},
which are necessary and sufficient if the problem satisfies certain constraint qualifications.
This approach results in a single optimization problem with complementarity constraints,
of which the decision problem is $\mathcal{NP}$-complete \cite{chung1989np}.
A specific form of the three-level problem is investigated in
\cite{dempe2020special}, where only the objective value of
the bottom-level problem appears in the objective functions of
the first and second levels.
If these conditions hold and all objectives and constraints are linear,
the problem can be reduced to a single level one
with complementarity constraints of polynomial size.
This model is similar to the worst-case reformulation
of (NORBiP) presented in \cref{sec:simplebilevel}.

Pessimistic bilevel problems for which no upper-level constraint
depends on lower-level variables are studied in \cite{zare2018class}.
The problem of finding an optimal solution to the pessimistic case is
shown to be $\mathcal{NP}$-hard, even if a solution to the optimistic
counterpart of the same problem is provided.
A variant is also defined, where the lower level may pick a
suboptimal response only impacting the upper-level objective.
This variant is comparable to the Objective-Robust Near-Optimal
Bilevel Problem defined in \cite{besanccon2019near}.
In \cite{Wiesemann2013}, the independent case of the pessimistic
bilevel problem is studied, corresponding to a special case of (NORBiP)
with $\delta=0$ and all lower-level constraints independent of the upper-level
variables. It is shown that the linear independent pessimistic
bilevel problem, and consequently the linear near-optimal robust
bilevel problem, can be solved in polynomial time while
it is strongly $\mathcal{NP}$-hard in the non-linear case.
\\

When the lower-level problem cannot be solved in polynomial
time, the bilevel problem is in general $\Sigma_2^P$-hard.
The notion of $\Sigma_2^P$-hardness and classes of the polynomial 
hierarchy are recalled in \cref{sec:simplebilevel}.
Despite this complexity result, new algorithms and corresponding implementations have
been developed to solve these problems and in particular,
mixed-integer linear bilevel problems
\cite{fischetti2017new,tahernejad2016branch,liu2020enhanced}.
Variants of the bilevel knapsack were investigated in 
\cite{caprara2014study}, and proven to be $\Sigma_2^P$-hard
as the generic mixed-integer bilevel problem.
\\

Multilevel optimization was initially investigated in \cite{jeroslow1985polynomial}
in the case of linear constraints and objectives at all levels. In this setting,
the problem is shown to be in $\Sigma_s^P$, with $s+1$ being the
number of levels.
The linear bilevel problem corresponds to $s=1$ and is in
$\Sigma_1^P\equiv\mathcal{NP}$. If at least the bottom-level
problem involves integrality constraints
(or more generally belongs to $\mathcal{NP}$ but not $\mathcal{P}$),
the multilevel problem with $s$ levels belongs to $\Sigma_s^P$.
A model unifying multistage stochastic and multilevel problems
is defined in \cite{bolusani2020unified}, based on a risk 
function capturing the component of the objective function which is
unknown to a decision-maker at their stage, either because
of a stochastic component or of
another decision-maker that acts after the current one.\\

As highlighted in \cite{bolusani2020unified,ralphscomplexity}, most
results in the literature on complexity of multilevel
optimization use $\mathcal{NP}$-hardness as the sole characterization.
This only indicates that a given problem is at least as hard as all
problems in $\mathcal{NP}$ and that no polynomial-time solution
method should be expected unless $\mathcal{NP} = \mathcal{P}$.\\

We characterize near-optimal robust multilevel problems
not only on the hardness or ``lower bound'' on complexity,
i.e. being at least as hard as all problems
in a given class but through their complexity ``upper bound'',
i.e. the class of the polynomial hierarchy they belong to.
The linear optimistic bilevel problem is for instance strongly
$\mathcal{NP}$-hard, but belongs to $\mathcal{NP}$
and is therefore not $\Sigma_2^P$-hard.

\section{Complexity of near-optimal robust bilevel problems}\label{sec:simplebilevel}

We establish in this section complexity results for near-optimal robust bilevel problems
for which the lower level $\mathcal{L}$ is a single-level problem parameterized by the
upper-level decision.
(NORBiP) can be reformulated by replacing each $k$-th semi-infinite
constraint \eqref{eq:semiinf} with the lower-level solution $z_k$ in
$\mathcal{Z}(x;\delta)$ that yields the highest value of $G_k(x, z_k)$:
\begin{subequations}\label{prob:norbip2}
\begin{align}
\min_{x,v}\,\, & F(x, v) \\
\text{s.t.}\,\,\,& \eqref{eq:upfeas}-\eqref{eq:lowopt}\\
& G_k(x, z_k) \leq 0 & \forall k \in \left[\![m_u\right]\!]\\
& z_k \in \argmax_{y\in\mathcal{Y}}\{G_k(x,y) \text{ s.t. } f(x, y) \leq f(x, v) + \delta, g(x, y) \leq 0\} & \forall k \in \left[\![m_u\right]\!]\label{eq:advprob}
\end{align}
\end{subequations}

From a game-theoretical perspective, the near-optimal robust
version of a bilevel problem can be seen as a three-player
hierarchical game. The upper level $\mathcal{U}$ and lower level $\mathcal{L}$
are identical to the original bilevel problem.
The third level is the adversarial problem $\mathcal{A}$ and selects
the worst near-optimal lower-level solution with respect to
upper-level constraints,
as represented by the embedded maximization in Constraint
\eqref{eq:advprob}. If the upper-level problem has multiple
constraints, the adversarial problem can be decomposed into problems
$\mathcal{A}_k, k \in \left[\![m_u\right]\!]$, where $m_u$ is the
number of upper-level constraints.
The interaction among the three players is depicted in
\cref{fig:norbib101}.
The blue dashed arcs represent a parameterization of the source
vertex by the decisions of the destination vertex, and the solid
red arcs represent an anticipation of the destination vertex decisions
in the problem of the source vertex. This convention will be used in
all figures throughout the paper.\\

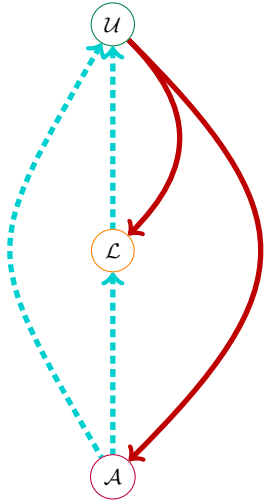
\begin{figure}[ht]
    \centering
\begin{tikzpicture}
	\begin{pgfonlayer}{nodelayer}
		\node [style={node_lower}] (0) at (0, 0) {$\mathcal{A}$};
		\node [style={node_style}] (1) at (0, 6) {$\mathcal{U}$};
		\node [style={node_intermediate}] (2) at (0, 3) {$\mathcal{L}$};
	\end{pgfonlayer}
	\begin{pgfonlayer}{edgelayer}
		\draw [style={edge_anticipate}] (1) to (2);
		\draw [style={edge_param}, in=45, out=-45, looseness=1.25] (1) to (2);
		\draw [style={edge_anticipate}, bend right, looseness=1.50] (1) to (0);
		\draw [style={edge_param}, bend left=45, looseness=1.50] (1) to (0);
		\draw [style={edge_anticipate}] (2) to (0);
	\end{pgfonlayer}
\end{tikzpicture}
    \caption{Near-optimal robust bilevel problem}
    \label{fig:norbib101}
\end{figure}

\noindent
The adversarial problem can be split into $m_u$ adversarial problems
as done in \cite{besanccon2019near}, each finding the worst-case with
respect to one of the upper-level constraints.
The canonical problem refers to the optimistic bilevel problem without
near-optimal robustness constraints.
We refer to the variable $v$ as the canonical lower-level decision.\\

The complexity classes of the polynomial hierarchy are only defined
for decision problems. We consider that an optimization problem belongs to
a given class if that class contains the decision problem of determining if there exists a
feasible solution for which the objective value at least as good as a
given bound.

\begin{definition}\label{def:star}
The decision problem associated with an optimization problem is in
$\mathcal{P}^*\left[\mathcal{H}\right]$, with $\mathcal{H}$ a set
of real-valued functions on a vector space $\mathcal{Y}$, iff:
\begin{enumerate}
\item it belongs to $\mathcal{P}$;
\item for any $h \in \mathcal{H}$, the problem with an additional linear
constraint and an objective function set as $h(\cdot)$ is also in $\mathcal{P}$.
\end{enumerate}
\end{definition}

A broad range of problems in $\mathcal{P}$ are also in
$\mathcal{P}^*\left[\mathcal{H}\right]$ for certain sets of functions
$\mathcal{H}$ (see \cref{ex:linear} for linear problems and
linear functions and \cref{ex:unimodular} for some combinatorial
problems in $\mathcal{P}$).
$\mathcal{NP}^*\left[\mathcal{H}\right]$ and
$\Sigma_s^{P*}\left[\mathcal{H}\right]$ are defined in a similar way.
We next consider two examples illustrating these definitions.

\begin{example}\label{ex:linear}
Denoting by $\mathcal{H}_L$ the set of linear functions from the space of
lower-level variables to $\mathbb{R}$, linear optimization
problems are in $\mathcal{P}^{*}\left[\mathcal{H}_L\right]$,
since any given problem with an additional linear constraint and a
different linear objective function is also a linear optimization problem.
\end{example}

\begin{example}\label{ex:unimodular}
Denoting by $\mathcal{H}_L$ the set of linear functions from the space of
lower-level variables to $\mathbb{R}$, combinatorial optimization
problems in $\mathcal{P}$ which can be formulated as linear optimization
problems with totally unimodular matrices are not in
$\mathcal{P}^{*}\left[\mathcal{H}_L\right]$ in general.
Indeed, adding a linear constraint may break the integrality of
solutions of the linear relaxation of the lower-level problem.
\end{example}

$\Sigma_s^{P}$ is the complexity class at the $s$-th level
of the polynomial hierarchy
\cite{stockmeyer1976polynomial,jeroslow1985polynomial},
defined recursively as
$\Sigma_0^{P} = \mathcal{P}$, $\Sigma_1^{P} = \mathcal{NP}$,
and problems of the class $\Sigma_s^{P}$, $s>1$ being solvable in
non-deterministic polynomial time, provided an oracle for
problems of class $\Sigma_{s-1}^{P}$.
In particular, a positive answer to a decision problem in $\mathcal{NP}$
can be verified, given a certificate, in polynomial time.
If the decision problem associated with an optimization problem is in $\mathcal{NP}$,
and given a potential solution, the objective value
of the solution can be compared to a given bound
and the feasibility can be verified in polynomial time.
We reformulate these statements in the following proposition:

\begin{proposition}\label{prop:oracle}
\cite{jeroslow1985polynomial}
An optimization problem is in $\Sigma_{s+1}^P$ if verifying that
a given solution is feasible and attains a given bound can be
done in polynomial time, when equipped with an oracle instantaneously
solving problems in $\Sigma_{s}^P$.
\end{proposition}

\noindent
\cref{prop:oracle} is the main property of the classes
of the polynomial hierarchy used to determine the complexity of
near-optimal robust bilevel problems in various settings throughout
this paper.

\begin{lemma}
Given a bilevel problem in the form of Problem \eqref{prob:bilevelstandardopt}, if
the lower-level problem is in $\mathcal{P}^*\left[\mathcal{H}\right]$, and
\begin{equation*}
-G_k(x,\cdot) \in \mathcal{H}\,\,\forall x, \forall k \in \left[\![m_u\right]\!],
\end{equation*}
then the adversarial problem \eqref{eq:advprob} is in $\mathcal{P}$.
\end{lemma}
\begin{proof}
The lower-level problem can equivalently be written in an epigraph form:
\begin{align*}
(v,w) \in \argmin_{y,u}\,\, & u\\
\text{s.t.}\,\,\,& f(x,y) - u \leq 0 \\
& g(x, y) \leq 0.
\end{align*}
Given a solution of the lower-level problem $(v,w)$ and an
upper-level constraint $G_k(x,y) \leq 0$, the adversarial
problem is defined by:
\begin{align*}
\min_{y,u}\,\, & -G(x,y)\\
\text{s.t.}\,\,\,& f(x,y) - u \leq 0\\
& g(x, y) \leq 0\\
& u \leq w.
\end{align*}
Compared to the lower-level problem, the adversarial problem contains
an additional linear constraint $u \leq w$ and an objective function
updated to $-G(x,\cdot)$.
\qed
\end{proof}

\begin{theorem}\label{theo:bilevelnp}
Given a bilevel problem $(P)$, if there exists $\mathcal{H}$
such that the lower-level problem is in
$\mathcal{NP}^*\left[\mathcal{H}\right]$ and
\begin{equation*}
-G_k(x, \cdot) \in \mathcal{H}\,\, \forall k \in \left[\![m_u\right]\!], \forall x \in \mathcal{X},
\end{equation*}

\noindent
then the near-optimal robust version of the bilevel problem is in 
$\Sigma_{2}^P$ like the canonical bilevel problem.
\end{theorem}
\begin{proof}
The proof relies on the ability to verify that a given solution $(x, v)$
results in an objective value at least as low as a bound $\Gamma$ according to
\cref{prop:oracle}. This verification can be carried out with the following steps: 

\begin{enumerate}
    \item Compute the upper-level objective value $F(x, v)$ and verify that $F(x, v) \leq \Gamma$;
    \item Verify that upper-level constraints are satisfied;
    \item Verify that lower-level constraints are satisfied;
    \item Compute the optimum value $\mathcal{L}(x)$ of the lower-level problem parameterized by $x$ and check if:
\begin{equation*}
    f(x,v) \leq \min_{y} \mathcal{L}(x);
\end{equation*}
    \item Compute the worst case:\\
\noindent
Find
\begin{equation*}
    z_k \in \argmax_{y\in\mathcal{Y}}\, \mathcal{A}_k(x, v)\,\,\, \forall k \in \left[\![m_u\right]\!];
\end{equation*}
\noindent
where $\mathcal{A}_k(x, v)$ is the $k$-th adversarial problem parameterized by $(x,v)$;
    \item Verify near-optimal robustness: $\forall k \in \left[\![m_u\right]\!]$,
    verify that the $k$-th upper-level constraint is feasible for the worst-case $z_k$.
\end{enumerate}

\noindent
Steps 1 and 2 can be carried out in polynomial time by assumption.
Step 3 requires to check the feasibility of a solution to a problem
in $\mathcal{NP}$. This can be done in polynomial time.
Step 4 consists in solving the lower-level problem,
while Step 5 corresponds to solving $m_u$ problems similar to the lower
level, with the objective function modified and an additional
linear constraint ensuring near-optimality.
\qed
\end{proof}

\begin{theorem}
Given a bilevel problem $(P)$, if the lower-level problem is convex and
in $\mathcal{P}^*\left[\mathcal{H}\right]$ with $\mathcal{H}$ a set of
convex functions, and if the upper-level constraints are such that
$-G_k(x,\cdot) \in \mathcal{H}$, then the near-optimal robust version
of the bilevel problem is in $\mathcal{NP}$.
If the upper-level constraints are convex non-affine with respect to
the lower-level constraints, the near-optimal robust version is in general
not in $\mathcal{NP}$.
\end{theorem}
\begin{proof}
If the upper-level constraints are concave with respect to
the lower-level variables, the adversarial problem defined as:
\begin{subequations}\label{prob:advconvex}
\begin{align}
\max_{y\in\mathcal{Y}} &\,\, G_k(x, y) \\
\text{s.t.}\,\,\,& g(x, y) \leq 0 \\
& f(x, y) \leq f(x, v) + \delta
\end{align}
\end{subequations}
is convex. Furthermore, by definition of $\mathcal{P}^*\left[\mathcal{H}\right]$, 
the adversarial problem is in $\mathcal{P}$.
\\

Applying the same reasoning as in the proof of
\cref{theo:bilevelnp}, Steps 1-3 are identical and can
be carried out in polynomial time.
Step 4 can be performed in polynomial time since $\mathcal{L}$ is in $\mathcal{P}$.
Step 5 is also performed in polynomial time since
$\forall k \in \left[\![m_u\right]\!]$, each $k$-th adversarial
problem \eqref{prob:advconvex} is a convex problem that can be
solved in polynomial time since $\mathcal{L}$ is in
$\mathcal{P}^*\left[\mathcal{H}\right]$.
Step 6 simply is a simple comparison of two quantities.\\

If the upper-level constraints are convex non-affine with respect to
the lower-level variables,
Problem \eqref{prob:advconvex} maximizes a convex non-affine
function over a convex set. Such a problem is $\mathcal{NP}$-hard in general.
Therefore, the verification that a given solution is feasible and
satisfies a predefined bound on the objective value requires solving the
$m_u$ $\mathcal{NP}$-hard adversarial problems.
If $\mathcal{L}$ is in $\mathcal{NP}^*\left[\mathcal{H}\right]$,
then these adversarial problems
are in $\mathcal{NP}$ by \cref{prob:norbip},
and the near-optimal robust problem is in $\Sigma_{2}^P$ according to
\cref{prop:oracle}.
\qed
\end{proof}




\section{Complexity of near-optimal robust mixed-integer multilevel problems}\label{sec:multilevel}

In this section, we study the complexity of a near-optimal
robust version of mixed-integer multilevel linear problems
(MIMLP), where the lower level itself is a
$s$-level problem and is $\Sigma_s^P$-hard.
The canonical multilevel problem is, therefore, $\Sigma_{s+1}^P$-hard
\cite{jeroslow1985polynomial}. For some instances of mixed-integer bilevel,
the optimal value can be approached arbitrarily
but not reached \cite{moore1990mixed}.
To avoid such pathological cases, we restrict our attention to
multilevel problems satisfying the criterion for mixed-integer bilevel
problems from \cite{fischetti2017new}:

\begin{property}\label{prop:mixed}
The continuous variables at any level $s$ do not appear in
the problems at levels that are lower than $s$ (the levels deciding after $s$).
\end{property}

More specifically, we will focus on mixed-integer multilevel
linear problems where the upper-most lower level $\mathcal{L}_1$
may pick a solution deviating from the optimal value, while
we ignore deviations of the levels $\mathcal{L}_{i>1}$.
This problem is noted ($\text{NOMIMLP}_s$) and depicted in \cref{fig:norbib101}.

\begin{figure}[h]
    \centering
\begin{tikzpicture}
	\begin{pgfonlayer}{nodelayer}
		\node [style={node_style}] (3) at (2.5, 9) {$\mathcal{U}$};
		\node [style={node_style}] (2) at (2.5, 6) {$\mathcal{L}_1$};
		\node [style={node_intermediate}] (1) at (0, 3) {$\mathcal{L}_2$};
		\node [style={node_lower}] (0) at (5, 3) {$\mathcal{A}$};
	\end{pgfonlayer}
	\begin{pgfonlayer}{edgelayer}
		\draw [style={edge_anticipate}] (3) to (2);
		\draw [style={edge_anticipate}] (2) to (1);
		\draw [style={edge_anticipate}] (3) to (1);
		\draw [style={edge_anticipate}] (3) to (0);
		\draw [style={edge_param}, in=30, out=-30, looseness=1.25] (3) to (2);
		\draw [style={edge_param}, in=90, out=-150, looseness=1.25] (3) to (1);
		\draw [style={edge_param}, in=90, out=-30, looseness=1.25] (3) to (0);
		\draw [style={edge_param}, in=15, out=-90, looseness=1.25] (2) to (1);
		\draw [style={edge_anticipate}] (2) to (0);
	\end{pgfonlayer}
\end{tikzpicture}
    \caption{Near-optimality robustness for multilevel problems}
    \label{fig:integer}
\end{figure}
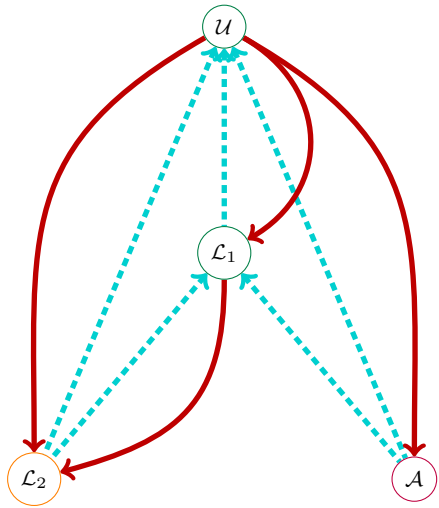
\noindent


The adversarial problem corresponds to a decision of the level
$\mathcal{L}_1$ different from the canonical decision.
This decision induces a different reaction from
the subsequent levels $\mathcal{L}_2$, $\mathcal{L}_3$.
Since the top-level constraints depend on the joint reaction
of all following levels, we will note $z_{ki} = (z_{k1}, z_{k2}, z_{k3})$
the worst-case joint near-optimal solution of all lower levels
with respect to the top-level constraint $k$.


\begin{theorem}\label{theo:NOMIMLP}
If $\mathcal{L}_1$ is in $\Sigma_{s}^{P*}\left[\mathcal{H}_L\right]$,
the decision problem associated with ($\text{NOMIMLP}_s$) is in
$\Sigma_{s+1}^P$ as the canonical multilevel problem.
\end{theorem}
\begin{proof}
Given a solution to all levels $(x_U,v_{1}, v_{2},... v_s)$ and a bound $\Gamma$,
verifying that this solution is (i) feasible, (ii) near-optimal robust of parameter
$\delta$, and (iii) has an objective value at least as good as the bound $\Gamma$ can
be done through the following steps:

\begin{enumerate}
    \item Compute the objective value and verify that it is lower than $\Gamma$;
    \item Verify variable integrality;
    \item Solve the problem $\mathcal{L}_1$, parameterized by $x_U$, and verify that the solution $(v_{1}, v_{2},...)$ is optimal;
    \item $\forall k \in \left[\![m_u\right]\!]$,
    solve the $k$-th adversarial problem. Let $z_k=(z_{k1}, z_{k2}... z_{ks})$ be the solution;
    \item $\forall k \in \left[\![m_u\right]\!]$, verify that the $k$-th upper-level
    constraint is feasible for the adversarial solution $z_k$.
\end{enumerate}

Steps 1, 2, and 5 can be performed in polynomial time.
Step 3 requires solving a problem in $\Sigma_{s}^P$, while step 4
consists in solving $m_u$ problems in $\Sigma_{s}^P$,
since $\mathcal{L}_1$ is in $\Sigma_{s}^{P*}\left[\mathcal{H}_L\right]$.
Checking the validity of a solution thus requires solving
problems in $\Sigma_{s}^P$ and is itself in $\Sigma_{s+1}^P$,
like the original problem.
\qed
\end{proof}

\section{Complexity of a generalized near-optimal robust multilevel problem}\label{sec:genmulti}

In this section, we study the complexity of a variant of the
problem presented in \cref{sec:multilevel} with $s+1$
decision-makers at multiple top levels
$\mathcal{U}_1, \mathcal{U}_2, ... \mathcal{U}_s$ and a
single bottom level $\mathcal{L}$.
We denote by $\mathcal{U}_1$ the top-most level.
We assume that the bottom-level entity may choose a solution
deviating from optimality.
This requires that the entities at all $\mathcal{U}_i\, \forall i \in \{1..s\}$
levels anticipate this deviation, thus solving a near-optimal
robust problem to protect their feasibility from it.
The variant, coined $\text{GNORMP}_s$, is illustrated in
\cref{fig:nomultigen}.
We assume throughout this section that \cref{prop:mixed} holds
in order to avoid the unreachability problem previously mentioned.
The decision variables of all upper levels are denoted by $x_{(i)}$,
and the objective functions by $F_{(i)}(x_{(1)}, x_{(2)}... x_{(s)})$.
The lower-level canonical decision is denoted $v$ as in previous sections.

\begin{figure}[ht]
\begin{minipage}{0.4\textwidth}
\centering
\begin{tikzpicture}
\begin{pgfonlayer}{nodelayer}
	\node [style={node_style}] (3) at (0, 9) {$\mathcal{U}_1$};
	\node [style={node_style}] (2) at (0, 6) {$\mathcal{U}_2$};
	\node [style={node_intermediate}] (1) at (0, 3) {$\mathcal{L}$};
	\node [style={node_lower}] (0) at (0, 0) {$\mathcal{A}$};
\end{pgfonlayer}
\begin{pgfonlayer}{edgelayer}
	\draw [style={edge_anticipate}] (3) to (2);
	\draw [style={edge_anticipate}] (2) to (1);
	\draw [style={edge_anticipate}] (1) to (0);
	\draw [style={edge_param}, in=45, out=-45, looseness=1.25] (3) to (2);
	\draw [style={edge_param}, in=120, out=-120, looseness=1.25] (3) to (1);
	\draw [style={edge_param}, in=120, out=-120, looseness=1.25] (3) to (0);
	\draw [style={edge_anticipate}, in=20, out=-20, looseness=1.25] (3) to (1);
	\draw [style={edge_anticipate}, in=20, out=-20, looseness=1.25] (3) to (0);
	\draw [style={edge_param}, in=45, out=-45, looseness=1.25] (2) to (1);
	\draw [style={edge_param}, in=45, out=-45, looseness=1.25] (2) to (0);
	\draw [style={edge_anticipate}] (2) to (0);
\end{pgfonlayer}
\end{tikzpicture}
\caption{Generalized NOR multilevel}\label{fig:nomultigen}
\end{minipage}
\begin{minipage}{0.4\textwidth}
\centering
\begin{tikzpicture}
\begin{pgfonlayer}{nodelayer}
	\node [style={node_style}] (3) at (2.5, 9) {$\mathcal{U}_1$};
	\node [style={node_style}] (2) at (2.5, 6) {$\mathcal{U}_2$};
	\node [style={node_intermediate}] (1) at (2.5, 3) {$\mathcal{L}$};
	\node [style={node_lower}] (0) at (0, 0) {$\mathcal{A}_2$};
	\node [style={node_lower}] (4) at (5, 0) {$\mathcal{A}_1$};
\end{pgfonlayer}
\begin{pgfonlayer}{edgelayer}
	\draw [style={edge_anticipate}] (3) to (4);
	\draw [style={edge_anticipate}] (2) to (4);
	\draw [style={edge_anticipate}] (1) to (4);
	\draw [style={edge_anticipate}] (3) to (0);
	\draw [style={edge_anticipate}] (3) to (2);
	\draw [style={edge_anticipate}] (2) to (1);
	\draw [style={edge_anticipate}] (1) to (0);
	\draw [style={edge_param}, in=45, out=-45, looseness=1.25] (3) to (2);
	\draw [style={edge_param}, in=120, out=-120, looseness=1.25] (3) to (1);
	\draw [style={edge_param}, in=70, out=-40, looseness=1.25] (3) to (4);
	\draw [style={edge_param}, in=45, out=-45, looseness=1.25] (2) to (1);
	\draw [style={edge_param}, in=45, out=-45, looseness=1.25] (2) to (0);
	\draw [style={edge_anticipate}] (2) to (0);
\end{pgfonlayer}
\end{tikzpicture}
\caption{Decoupling of the two adversarial problems}\label{fig:nomultigen2}
\end{minipage}
\end{figure}

If the lowest level $\mathcal{L}$ belongs to $\mathcal{NP}$,
$\mathcal{U}_s$ belongs to $\Sigma_{2}^P$ and the original
problem is in $\Sigma_{s+1}^P$. In a more general multilevel case, if
the lowest level $\mathcal{L}$ solves a problem in $\Sigma_{r}^P$,
$\mathcal{U}_s$ solves a problem in $\Sigma_{r+1}^P$ and $\mathcal{U}_1$ in
$\Sigma_{r+s}^P$. \\

We note that for all fixed decisions
$x_{(i)} \forall i \in \{1..s-1\}$,
$\mathcal{U}_s$ is a near-optimal robust bilevel problem.
This differs from the model presented in \cref{sec:multilevel}
where, for a fixed upper-level decision, the top-most lower level
$\mathcal{L}_1$ is the same parameterized problem as in the canonical
setting. Furthermore, as all levels $\mathcal{U}_i$ anticipate
deviations of the lower-level decision in the near-optimal set,
the worst case can be formulated with respect to the constraints
of each of these levels.
In conclusion, distinct adversarial problems
$\mathcal{A}_{i}\, \forall i \in \{1..s\}$ can be formulated.
Each upper level $\mathcal{U}_i$ integrates the reaction of the
corresponding adversarial problem in its near-optimality
robustness constraint.
This formulation of $\text{GNORMP}_s$ is depicted
\cref{fig:nomultigen2}.

\begin{theorem}\label{theo:gnor}
Given a $s+1$-level problem $\text{GNORMP}_s$, if the bottom-level
problem parameterized by all upper-level decisions
$\mathcal{L}(x_{(1)}, x_{(2)}... x_{(s)})$ is in $\Sigma_{r}^{P*}$,
then $\text{GNORMP}_s$ is in $\Sigma_{r+s}^P$
like the corresponding canonical bilevel problem.
\end{theorem}

\begin{proof}
We denote by $x_U = (x_{(1)}, x_{(2)},... x_{(s)})$
and $m_{U_i}$ the number of constraints of problem $\mathcal{U}_i$.
As for \cref{theo:bilevelnp},
this proof is based on the complexity of
verifying that a given solution $(x_U, v)$
is feasible and results in an objective value below a given bound.
The verification requires the following steps:

\begin{enumerate}
    \item Compute the top-level objective value and assert that it is below the bound;
    \item Verify feasibility of $(x_U, v)$ with respect to the
    constraints at all levels;
    \item Verify optimality of $v$ for $\mathcal{L}$ parameterized by $x_U$;
    \item Verify optimality of $x_{(i)}$ for the near-optimal robust problem
    solved by the $i$-th level
    $\mathcal{U}_i(x_{(1)}, x_{(2)}... x_{(i-1)}; \delta)$ parameterized
    by all the decisions at levels above and the near-optimality tolerance $\delta$;
    \item Compute the near-optimal lower-level solution $z_{k}$ which is the
    worst-case with respect to the $k$-th constraint of the top-most level $\forall k \in \left[\![m_{U_1}\right]\!]$;
    \item Verify that each $k \in \left[\![m_{U_1}\right]\!]$ top-level constraint is satisfied with respect to the corresponding worst-case solution $z_k$.
\end{enumerate}

Steps 1-2 are performed in polynomial time. Step 3 requires solving
Problem $\mathcal{L}(x_U)$, belonging to $\Sigma_r^P$.
Step 4 consists in solving a generalized near-optimal robust multilevel
problem $\text{GNORMP}_{s-1}$ with one level less than the current problem.
Step 5 requires the solution of $m_{U_1}$ adversarial problems
belonging to $\Sigma_r^P$ since $\mathcal{L}$ is in $\Sigma_r^{P*}$.
Step 6 is an elementary comparison of two quantities for each $k\in \left[\![m_u\right]\!]$.
The step of highest complexity is Step 4.
If it requires to solve a problem in $\Sigma_{r+s-1}^P$, then
$\text{GNORMP}_{s}$ is in $\Sigma_{r+s}^P$ similarly to its canonical problem.\\

Let us assume that Step 4 requires to solve a problem outside $\Sigma_{r+s-1}^P$.
Then $\text{GNORMP}_{s-1}$ is not in $\Sigma_{r+s-1}^P$ as the associated
canonical problem, and that Step 4 requires to solve a problem not in
$\Sigma_{r+s-2}^P$. By recurrence, $\text{GNORMP}_{1}$ is not
in $\Sigma_{r+1}^P$.
However, $\text{GNORMP}_{1}$ is a near-optimal robust bilevel problem
where the lower level itself in $\Sigma_{r}^P$; this corresponds to
the setting of \cref{sec:multilevel}.
This contradicts \cref{theo:NOMIMLP}, $\text{GNORMP}_{s-1}$
is therefore in $\Sigma_{r+s-1}^P$.
Verifying the feasibility of a given solution to $\text{GNORMP}_{s}$ requires solving
a problem at most in $\Sigma_{r+s-1}^P$. Based on \cref{prop:oracle},
$\text{GNORMP}_{s}$ is in $\Sigma_{r+s}^P$ as its canonical multilevel problem.
\qed
\end{proof}

In conclusion, \cref{theo:gnor} shows that adding near-optimality robustness
at an arbitrary level of the multilevel problem does not increase
its complexity in the polynomial hierarchy.
By combining this property with the possibility to add near-optimal deviation
at an intermediate level as in \cref{theo:NOMIMLP}, near-optimality
robustness can be added at multiple levels of a multilevel model
without changing its complexity class in the polynomial hierarchy.

\section{Conclusion}\label{sec:conclusioncomplex}

In this paper, we have shown that for many configurations of bilevel
and multilevel optimization problems, adding near-optimality
robustness to
the canonical problem does increase its complexity in the polynomial 
hierarchy. This result is obtained even though
near-optimality robustness constraints add another level to the
multilevel problem, which in general would change the complexity class.
We defined the class $\Sigma_s^{P*}\left[\mathcal{H}\right]$ that is
general enough to capture many non-linear multilevel problems
but avoids corner cases where the modified objective or
additional linear constraint makes the problem harder to solve.\\

Future work will consider specialized solution algorithms for
some classes of near-optimal robust bilevel and multilevel problems.

\begin{acknowledgements}
The work of Mathieu Besançon was supported by the Mermoz scholarship and the GdR RO.
\end{acknowledgements}

\bibliographystyle{ieeetr}
\bibliography{refs}

\end{document}